\documentclass[12pt]{amsart}
\usepackage[margin=1in]{geometry}
\usepackage{newtxtext}
\usepackage{newtxmath}
\usepackage[alphabetic]{amsrefs}
\usepackage{graphicx}
\usepackage{caption}
\usepackage{tikz}
	\usetikzlibrary{decorations.pathreplacing}
	\usetikzlibrary{decorations.pathmorphing}
	\usetikzlibrary{patterns}
\usepackage{setspace}
\usepackage[hidelinks]{hyperref}

\newcommand{\cd}{\cdot}
\newcommand{\ra}{\rightarrow}
\newcommand{\pr}{\prime}
\newcommand{\de}{\partial}
\newcommand{\te}{\theta}

\newcommand{\Q}{\mathbb{Q}}
\newcommand{\R}{\mathbb{R}}
\newcommand{\Z}{\mathbb{Z}}
\newcommand{\abs}[1]{\left\lvert #1 \right\rvert}

\newcommand{\lbar}[1]{\overline{#1}}

\newtheorem{theorem}{Theorem}
\newtheorem*{theorem*}{Theorem}

\newtheorem{lemma}{Lemma}

\theoremstyle{definition}
\newtheorem{definition}{Definition}
\theoremstyle{remark}
\newtheorem*{note}{Note}

\newtheorem{fact}{Fact}

\begin{document}

\title{Enhanced Hantzsche Theorem}
\date{\today}

\author{Michael H. Freedman}
\address{\hspace{-\parindent}Michael H. Freedman}
\email{mfreedman@cmsa.fas.harvard.edu}

\begin{abstract}
	A closed 3-manifold $M$ may be described up to some indeterminacy by a Heegaard diagram $\mathcal{D}$. The question ``Does $M$ smoothly embed in $\R^4$?'' is equivalent to a property of $\mathcal{D}$ which we call \emph{doubly unlinked} (DU). This perspective leads to an enhancement of Hantzsche's embedding obstruction.
\end{abstract}

\maketitle

\section{Introduction}
We work entirely in the smooth category. It is well known that every closed 3-manifold $M$ embeds in $\R^5$. Which $M^3$ embed in $\R^4$ is an area of current activity \cite{bb22} and also has a long history \cite{hantzsche37}. Recently, \cite{m22} used gauge theory to construct a 3D integral homology sphere $X$ which, although it embeds in some 4D integral homology sphere, does \emph{not} embed in $S^4$. Although gauge theory and (relatedly) Floer homology are powerful tools, they do not appear to distinguish a homotopy 4-sphere from $S^4$ ($:= S^4_{\text{std}}$), and indeed \cite{m22} shows that $X$ does not embed in any homotopy 4-sphere.

It is a natural thought that to distinguish $S^4$ from some other homotopy 4-spheres (if they exist), that the product structure
\begin{equation}\label{eq:product-structure}
	\R^3 \times \R \cong \R^4 \cong S^4 \setminus \text{pt}
\end{equation}
could be useful. In fact, among existing tools, Khovanov homology and in particular the Rassmussen $s$-invariant stand out because they \emph{may} depend on, and at least the definition exploits, the product structure \eqref{eq:product-structure}. The program of \cite{fgmw10} and \cite{mp23}, still open, was to use a version of the $s$-invariant to distinguish some homotopy ball $\mathcal{B}^4$ from $B^4$ in terms of the slice genus of some knot $K \subset S^3$ in their common boundary. The foliation by level sets may elucidate relative embeddings of surfaces.

In this note, we replay the same idea: use the foliated structure \eqref{eq:product-structure} to find some condition which $M^3 \hookrightarrow S^4$ must satisfy which would not necessarily obstruct an embedding of $M^3$ in the general homotopy 4-sphere. Both programs remind us of nuclear physics; an object (a homotopy 4-sphere) is studied by shooting something smaller (surfaces or now 3-manifolds) at it to see what sticks. The new condition is on a Heegaard diagram $\mathcal{D}$ for the 3-manifold $M$, defined below. We are successful in finding a condition, doubly unlinked (DU), defined below, which is equivalent to the embedding problem $M^3 \overset{?}{\hookrightarrow} S^4$. Where we fall short is in proposing a new tool to analyze this new condition; we have found nothing like Khovanov homology to make computations. The condition DU is manifestly a strong one, since it is ``rare'' that $M^3$ embeds in $S^4$. In section \ref{sec:weak-unlink}, we prove an enhanced Hantzsche's theorem. It has the same conclusion on the linking forms but requires an input weaker than an embedding; the condition ``unlink'' within DU is replaced by ``linking numbers vanish.'' Hantzsche's theorem says that if $M^3$ embeds in $S^4$, its linking form LF on $\operatorname{tor}(H_1(M^3,\Z))$ vanishes on a $\surd$-order subgroup, i.e.\ MF$(M)$ is hyperbolic; see section \ref{sec:weak-unlink}.

The author would like to thank Cameron Gordon and Slava Krushkal for helpful conversations.

To fix notation, we review some basic surface theory. The handlebody $H_g$ is the oriented 3-manifold obtained by attaching $g$ 1-handles to the 3-ball. The co-cores of the 1-handles are 2-disks which meet $\de H_g =: \Sigma_g$ in $g$ pairwise-disjoint, homologically independent, simple closed curves (scc). A set of $g$ pairwise-disjoint homologically independent scc is called a \emph{geometric Lagrangian} (GL) on $\Sigma_g$. The mapping class group of the handlebody, MCG$(H_g)$, also called the \emph{handlebody group}, acts transitively on the GLs that bound disjoint disks in $H_g$. Two GLs on a surface are \emph{handlebody equivalent} iff they determine the same HB under handle attachment, equivalently if the handlebody group of either GL has the other GL in its orbit. More generally, MCG$(\Sigma_g)$ acts transitively on the set of all GL. A \emph{Heegaard diagram} for $M$ consists of two GLs, one drawn in red and the other in blue. The red $R$ (blue $B$) GL determines an upper (lower) $\mathrm{HB}_u$ ($\mathrm{HB}_\ell$) boundary $\Sigma_g$ ($-\Sigma_g$) with $M = \mathrm{HB}_\ell \smile_{\Sigma_g} -\mathrm{HB}_u$.

More generally, if handlebodies $\mathrm{HB}_1$ and $\mathrm{HB}_2$ are glued by $f: \de \mathrm{HB}_1 \ra \de \mathrm{HB}_2$, $f$ and $h_2\big\vert_{\de \mathrm{HB}_2} \circ f \circ h_1\big\vert_{\de\mathrm{HB}_1}$ determine the same 3-manifold for $h_1 \in \operatorname{MCG}(\mathrm{HB}_1)$ and $h_2 \in \operatorname{MCG}(\mathrm{HB}_2)$. That, and bar-stabilization, defined below, generate the full indeterminacy, up to isotopy, of the Heegaard decomposition (HD) of a fixed $M^3$ (via the classical Reidemeister-Singer theorem \cite{lau14}).

\begin{definition}
	A Heegaard diagram $\mathcal{D}$ on $\Sigma_g$ is \emph{doubly unlinked} iff there exists an embedding $e: \lbar{\Sigma} \hookrightarrow S^3$ of a bar-stabilization of $\Sigma_g$ so that the bar-stabilizations $\lbar{R}$ and $\lbar{B}$ can be mapped (separately) to 0-framed unlinks $e(h_\ell(\lbar{R}))$ and $e(h_u(\lbar{B}))$ under the actions of the lower (upper) handlebody groups, $h_\ell \in \mathrm{HG}_\ell$ and $h_u \in \mathrm{HG}_u$.
\end{definition}

To connect to the algebra, $H_1(\Sigma_g;\Z) \cong Z^{2g}$ may be given a standard basis $(x_1,\dots,x_g,p_1,\dots,p_g)$ of ``meridians'' and ``longitudes'' with a symplectic intersection form $\omega$:
\begin{equation}
	\begin{split}
		& x_i \cd x_j = p_i \cd p_j = 0 \\
		& x_i \cd p_j = \delta_{ij},\ p_j \cd x_i = -\delta_{ij}
	\end{split}
\end{equation}
Considering  the action of MCG$(\Sigma_g) =: \mathrm{MCG}_g$ on $H_1(\Sigma_g;\Z)$, one obtains an epimorphism:
\begin{equation}\label{eq:epimorphism}
	1 \ra \mathcal{I}_g \ra \mathrm{MCG}_g \ra \operatorname{SP}(2g;\Z) \ra 1,
\end{equation}
the group of symplectic automoprhisms of $\Z^{2g}$. The map is easily seen to be onto by lifting the Burkhardt generators of SP$(2g,\Z)$ back to rather easy surface automorphisms, see \cite{fm11}.

In the algebraic context, a Lagrangian is a sublattice $L$, with (1) $\Z^{2g} \slash L \cong \Z^g$, and (2) $\omega \equiv 0$ on $L$. It follows from \cite{fm11} that every alegebraic $L$ lifts to a GL $\mathcal{L}$ with indeterminacy given by the action of the Torelli group $\mathcal{I}_g$, the kernel on \eqref{eq:epimorphism}.

\subsection*{Stabilization}
Now consider stabilization of genus $\Sigma_g \ra \Sigma_{g+k}$ and attendant stabilizations $\mathrm{MCG}_g$ $\ra \mathrm{MCG}_{g+k}$ of the mapping class groups. We use two notions for stabilizing:

\begin{align}
	& \mathrm{MCG}_g \xrightarrow{-} \mathrm{MGC}_{g+k}, \text{ and} \\
	& \mathrm{MCG}_g \xrightarrow{\wedge} \mathrm{MCG}_{g+k+k^\pr}
\end{align}
The bar, $f \mapsto \lbar{f}$, means isotope $f$ to be fixed near some point $p$ of $\Sigma_g$, at $p$ form $\Sigma_g \# \Sigma_k$ and extend $f$ over $\Sigma_k^-$ via the standard geometric symplectic map:
\begin{equation}
	\begin{array}{rcccccc}
		\multicolumn{1}{l}{} & m_1 & & m_k & \ell_1 & & \ell_k \\
		\multicolumn{1}{r|}{m_1} & & & \multicolumn{1}{l|}{} & -1 & & \multicolumn{1}{l|}{} \\
		\multicolumn{1}{r|}{} & & \scalebox{1.5}{0} & \multicolumn{1}{l|}{}  & & \ddots & \multicolumn{1}{l|}{} \\
		\multicolumn{1}{r|}{m_k} & & & \multicolumn{1}{l|}{} & & & \multicolumn{1}{l|}{-1} \\ \cline{2-7} 
		\multicolumn{1}{r|}{\ell_1} & 1 & & \multicolumn{1}{l|}{} & & & \multicolumn{1}{l|}{}   \\
		\multicolumn{1}{r|}{} & & \ddots & \multicolumn{1}{l|}{} & & \scalebox{1.5}{0} & \multicolumn{1}{l|}{}   \\
		\multicolumn{1}{r|}{\ell_k} & & & \multicolumn{1}{l|}{1} & & & \multicolumn{1}{l|}{}  
	\end{array}
\end{equation}
exchanging meridians and longitudes with the indicated signs. Sometimes we write $\lbar{\Sigma}$ for $\Sigma_g \# \Sigma_k$, treating $k$ as variable. Similarly, $\widehat{f}$ is an automorphism of $\Sigma_g \# \Sigma_k \# \Sigma_{k^\pr} =: \widehat{\Sigma}$, with the extension as above over $\Sigma_k^-$ but via the \emph{identity} on $\Sigma_{k^\pr}^-$.

Recalling that $\de H_g = \Sigma_g$, $f: \Sigma_g \ra \Sigma_g$ determines a closed 3-manifold
\begin{equation}
	M_f \cong H_g \smile_f -H_g.
\end{equation}
In these terms, bar-stabilization stabilizes the Heegaard decompositions of $M_f$,
\begin{equation}
	M_{\lbar{f}} \cong M_f,
\end{equation}
whereas
\begin{equation}\label{eq:stabilize-manifold}
	M_{\widehat{f}} \cong M_f \#(\#_{k^\pr} S^1 \times S^2),
\end{equation}
stabilizes the manifold as well. A Heegaard decomposition $\mathcal{D}$ is bar-stabilized by adding $\{m_1,\dots,m_k\}$ to $R$ and $\{\ell_1,\dots,\ell_k\}$ to $B$. Similarly, $\hat{\phantom{\_}}$-stabilization further adds $\{m_{k+1},\dots,m_{k+k^\pr}\}$ to $\lbar{R}$ and another copy of $\{m_{k+1},\dots,m_{k+k^\pr}\}$ to $\lbar{B}$.

Related to the definition of doubly unlinked (DU), we have:

\begin{definition}
	$f \in \operatorname{MCG}(\Sigma_g)$ is \emph{unlink preserving} (UP) iff there exists a GL $\lbar{\mathcal{L}} \subset \lbar{\Sigma}_g$, a bar-stabilization of $\Sigma_g$ (some $k \geq 0$), and a choice of embedding $\lbar{\Sigma} \overset{e}{\hookrightarrow} S^3$ into the 3-sphere so that under this embedding, both $e(\lbar{\mathcal{L}})$ and $e(\lbar{f}(\lbar{\mathcal{L}}))$ are zero-framed unlinks. The framings are determined by the normal to $\lbar{\mathcal{L}}$ in $\lbar{\Sigma}$. If $f$ is not UP we say it is \emph{unlink busting} (UB). We will sometimes drop $e$ from the notation.
\end{definition}

\begin{note}
	The properties UP and UB depend only on the conjugacy class of $f$ in $\mathrm{MCG}_g$.
\end{note}

\begin{theorem}\label{thm:ub-up}
	Let $M$ be the closed 3-manifold with Heegaard diagram $\mathcal{D}$. Then $M$ (smoothly) embeds in $\R^4$ iff some bar-stabilization of $\mathcal{D}$ is DU. In terms of a gluing map $f$, $f$ is UP iff for some conjugate $\lbar{f}^g$ of some bar-stabilization of $f$, $\mathrm{HB} \smile_{\lbar{f}^g} -\mathrm{HB}$ (smoothly) embeds in $\R^4$.
\end{theorem}

Before giving the proof, we consider some simple examples to fix concepts. The ``round'' embedding $S^3 \subset \R^4$ has only two critical points (w.r.t., say, the 4th coordinate); the Heegaard surface is $S^2$. There is not much to say about Lagrangians. Next, consider $S^3 \hookrightarrow \R^4$ with one critical point of each index $=0$, 1, 2, and 3. The corresponding Heegaard decomposition is genus 1, see Figure \ref{fig:one-critical-pt}. Let's call the boundary of the co-core of the 1-handle the meridian $x$, and the attaching circle $p$ of the 2-handle the longitude. In the language of Heegaard decompositions, the ``red'' GL is $\{x\}$ and the ``blue'' GL is $\{p\}$. Examining the Morse function tells us that both are 0-framed and unknotted (again, the framing is by the normal to the surface $S^3 \cap \R^3 \times 0 =: \Sigma_1$, living in the middle ``0,'' level of $\R^4$). For an embedded torus in $\R^3 \times 0$ exactly two classes of scc are 0-framed, so they must be $x$ and $p$. Since $p$ is unknotted, the picture in the middle level is completely standard.

\begin{figure}[ht]
	\centering
	\begin{tikzpicture}[scale=0.8]
		\draw (0,0) ellipse (6 and 3);
		\node at (6.4,0) {$\Sigma_1$};
		\draw (0,0) ellipse (4 and 1.75);
		\draw (2.2,0.15) arc (0:-180:2.2 and 0.6);
		\draw (2,-0.1) arc (0:180:2 and 0.6);
		\node at (4.4,0) {$p$};
	
		\draw (0,3) arc (90:-90:0.4 and 1.25);
		\draw[dashed] (0,3) arc (90:270:0.4 and 1.25);
		\node at (0.7,2.3) {$x$};
		\node at (-6.3,0) {$\hphantom{m}$};
	\end{tikzpicture}
	\caption{}\label{fig:one-critical-pt}
\end{figure}

In $(x,p)$-basis the Heegaard gluing map induced by the height function of the embedded $S^3$ is
\[
	f = \text{rotation} = \begin{array}[b]{rccc}
		\multicolumn{1}{l}{} & x & p \\
		\multicolumn{1}{r|}{x} & 0 & \multicolumn{1}{c|}{1} \\
		\multicolumn{1}{r|}{p} & -1 & \multicolumn{1}{c|}{0} \\
	\end{array}\ ,
\]
a rotation.

But suppose, instead, that starting with $\Sigma_1 = \de H_1$, we asked if $M_f$ embeds in $\R^4$ for $f$ the shear map
\[
	f = \text{shear} = \begin{array}[b]{rccc}
		\multicolumn{1}{l}{} & x & p \\
		\multicolumn{1}{r|}{x} & 1 & \multicolumn{1}{c|}{0} \\
		\multicolumn{1}{r|}{p} & 1 & \multicolumn{1}{c|}{1} \\
	\end{array}\ ?
\]
This map sends $x$ to $x+p$, the diagonal. The diagonal is unknotted, but its framing $=1$; there is no local model for pinching it off by an ambient 1-surgery. Does this mean we have found an obstructing to embedding $M_{\text{shear}}$? We better not have as
\begin{equation}
	M_{\text{shear}} \cong S^3 \cong M_{\text{rotation}}
\end{equation}

The puzzle is resolved by remembering that for the ``only if'' direction the two GLs need only be \emph{handlebody equivalent} to a 0-framed unlink, this being part of the indeterminacy of a $\mathcal{D}$ associated to $M$. In this example, one GL is $x$, a 0-framed unknot, and the other GL is the diagonal which \emph{becomes} a 0-framed unknot after a Dehn twist along the essential disk of $H_1$.

So we should remember MCG$(\text{HB})$ acts on the GLs, allowing them to be modified by such Dehn twists, as well as handle slides and ``half twists,'' see Figure \ref{fig:half-twists}, which together generate MCG$(\text{HB})$.

\begin{note}
	As  explained in \cite{fm11}, without half twists, Dehn-twists and handle slides only generate the index 2 subgroup $K$, below:
	\begin{equation}
		\operatorname{MCG}(\mathrm{HB}) \ra \operatorname{Out}(\operatorname{free}(g)) \xrightarrow{\operatorname{det}} \Z_2,
	\end{equation}
	$K = $ kernel of composition.
\end{note}

\begin{figure}[ht]
	\centering
	\begin{tikzpicture}
		\node at (0,1.5) {$\operatorname{MCG}(\mathrm{HB}) \ra \operatorname{Out}(\operatorname{free}(g)) \xrightarrow{\mathrm{det}} \Z_2$};
		\node at (0,0.5) {$K = $ kernel of competition};
		\node at (-3.8,-1) {\includegraphics[width=5cm]{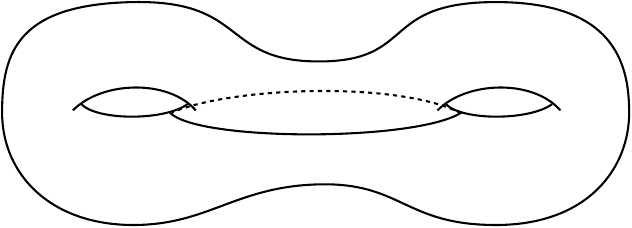}};
		\node at (0,-1) {$\xrightarrow{\text{half twist}}$};
		\node at (3.8,-1) {\includegraphics[width=5cm]{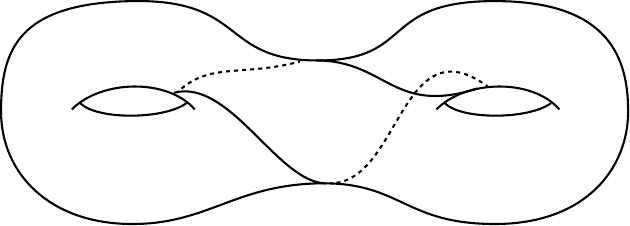}};
	
		\path[pattern=north east lines,pattern color = gray] (-3.75,-1) ellipse (1.1 and 0.15);
		\path[pattern=north east lines,pattern color = gray] (2.7,-0.85) to[out=20,in=175] (3.9,-1.55) to[out=0,in=225] (4.6,-0.85) to[out=-15,in=200] (5.15,-0.8) to[out=120,in=50] (4.6,-0.85) to[out=170,in=0] (3.8,-0.6) to[out=200,in=35] (2.7,-0.85);
	\end{tikzpicture}
	\caption{}\label{fig:half-twists}
\end{figure}

The handlebody groups are subtle, but Wajnryb found a finite representation \cite{waj98}.

To get a bit more intuition, let us draw a picture suggesting an insurmountable framing difficulty (confirmed by Hantzsche's theorem) to embedding the lens space $L_{3,1}$.

\begin{figure}[ht]
	\centering
	\begin{tikzpicture}
		\node at (0,0) {\includegraphics[scale=0.75]{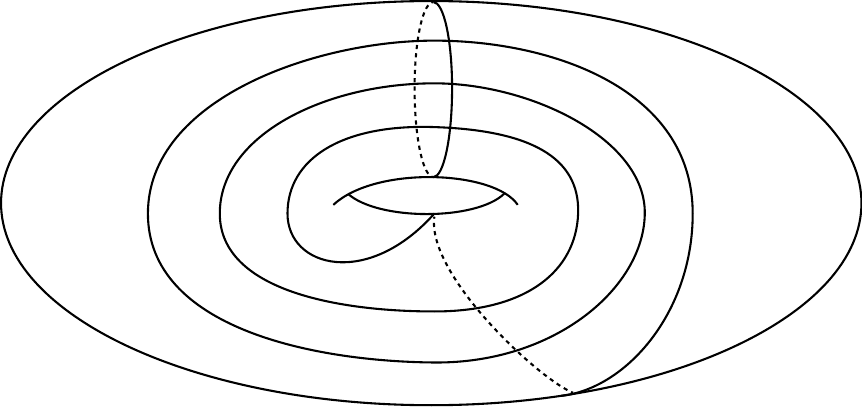}};
		\node at (-2.1,3.1) {Lagrangian $L$};
		\draw[->] (-0.8,3.1) -- (-0.1,2.8);
		\node at (2.5,3.1) {Lagrangian $f(L)$};
		\draw[->] (2.5,2.8) -- (2.2,1.9);
	\end{tikzpicture}
	\caption{}
\end{figure}

Zeeman \cite{zee65}, as part of his study of twist-spun knots, constructed explicit embeddings of all punctured $p = $ odd, lens spaces $L^-_{p,q} \hookrightarrow S^4$ and therefore embeddings $L_{p,q} \# -L_{p,q} \hookrightarrow S^4$. The associated DU Heegaard diagrams are a bit difficult to draw, but their existence, guaranteed by Theorem \ref{thm:ub-up}, will be enough to produce a partial converse to Hantzsche's theorem, our Theorem \ref{thm:enhanced-hantz}.

\begin{proof}[Proof of Theorem \ref{thm:ub-up}]
	($\Leftarrow$) We use the language of Heegaard decompositions. Start with an embedding $e: \lbar{\Sigma} \hookrightarrow S^3 \times 0$, for some bar-$k$-stabilization of $\Sigma_g$ with geometric Lagrangians $R$ and $B$ (red and blue) so that both $e(R)$ and $e(B)$ are zero framed $g+k$-component unlinks. Thus, $e(R)$ bounds $g+k$ disjoint properly embedded 2-disks $(\coprod_{i=1}^{g+k} D^2,\de) =: (\lbar{\mathcal{D}}, \de \lbar{\mathcal{D}}) \hookrightarrow (S^3 \times 0, e(\lbar{\Sigma}))$. By the framing condition we may assume the collar of $\de \lbar{D}$ is normal to $e(\lbar{\Sigma})$ and that $\lbar{\mathcal{D}} \pitchfork e(\lbar{\Sigma})$, all intersections are transverse. Indeed, since $R \subset \lbar\Sigma$ is complete, $\operatorname{int}(\lbar{\mathcal{D}}) \cap e(\lbar{\Sigma})$ consists of disjoint embedded scc in $\lbar{\mathcal{D}}$, which are either trivial in $e(\lbar{\Sigma})$ or lie in the planar domain $e(\lbar\Sigma \setminus R)$.

	Start to evolve $e(\lbar{\Sigma})$ downward in the product structure of $S^3 \times [0,-\infty)$ As it moves downward, carry out ambient 1-surgeries on innermost circles of $\operatorname{int}(\lbar{\mathcal{D}} \cap e(\lbar{\Sigma}))$ until none remain. Then, finally, do ambient 1-surgery on the components of $R$. All 1-surgeries use disks or subdisks of $\lbar{\mathcal{D}}$ as their cores, with the surgery traces built around those. What is the entire trace of those ambient surgeries? It is a copy of $(\mathrm{HB}_R \setminus \mathrm{balls})$, the canonical (punctured) handlebody $\mathrm{HB}_R$ associated to $R$ with one additional (open) 3-ball deleted for each scc of $\operatorname{int}(\mathcal{D} \cap e(\lbar{\Sigma}))$. Surgery on trivial components produce additional 2-sphere boundaries directly, and surgery on parallels to band sums $R$, produce a 2-sphere eventually, certainly by the time $R$ is surgered. Finally, these 2-spheres lie in $S^3$-levels and may be filled in by 3-balls to produce an embedding of $\mathrm{HB}_R \subset S^3 \times [0,-\infty)$.

	Similarly, we take embedded disks in $S^3 \times 0$ bounding $e(B)$ transverse to $e(\lbar{\Sigma})$ and surger all intersections while evolving $e(\lbar{\Sigma})$ upwards in the product structure of $S^3 \times [0,\infty)$. The trace of these surgeries (with additional 2-sphere surgeries to fill the resulting $S^2$-holes) will be diffeomorphic rel boundary to $(\mathrm{HB}_B,\lbar{\Sigma})$. The union of the upper and lower handlebodies is diffeomorphic to $M$.

	($\Rightarrow$) This direction requires a well-known tool, ``ambient'' Morse theory; see the appendix to \cite{af21} for a succinct exposition. From the next statement, we only need the case where $N$ is $S^3$ and $M$ a 3-manifold.

	\begin{theorem*}[Ambient Morse Theorem]
		Let $M \xrightarrow{g} N \times R$ be an embedding where $M,N$ are closed manifolds and $\dim(N) \geq \dim(M)$. First, $g$ may be perturbed by a small ambient isotopy so that $\pi_2 \circ g$ has only Morse critical points, $\pi_2: N \times R \ra R$ being the projection. Second, a further isotopy of $g$ (no longer small) may be found which preserves the set of critical points and their indices while resulting in an ``ordered'' Morse function in the sense that $\pi_2 \circ g(c) > \pi_2 \circ g(c^\pr)$ if $\operatorname{index}(c) \geq \operatorname{index}(c^\pr)$, $c$ and $c^\pr$ critical points. Finally, if $\dim(N) > \dim(M)$, the critical points of each index may have their heights ($R$-coordinates) rearranged among themselves arbitrarily. However, if $\dim(N) = \dim(M)$, this will generally \emph{not} be possible; critical points of a fixed index may nest and thus have an intrinsic ordering.
	\end{theorem*}

	Suppose $M^3$ embeds in $S^4$, or equivalently, in $S^3 \times R$. This theorem allows us to change the initial embedding by an isotopy so that the $R$-coordinate is an ordered Morse function on $M^3$. Let $0 \subset R$ be a generic level above the 0-, 1-handles (critical points) and below the 2-, 3-handles. Let us suppress the notation for this improved embedding and simply write $M^3 \hookrightarrow S^3 \times R$.

	$M^3 \cap S^3 \times 0 =: \Sigma$, a Heegaard surface for $M^3$ with $M^3 \cap S^3 \times [0,-\infty)$ and $M^3 \cap (S^3 \times [0,+\infty))$ being the lower and upper handlebodies. To find a GL for the lower HB look at the ascending manifolds of the critical points of index $=1$. These are disjointly embedded 2-disks meeting $\Sigma$ in a 1-manifold $\Gamma$ which contains a GL $\mathcal{L}_- \subset \Sigma$. Some of the circle components of $\Gamma$ come from 1-handles joining disconnected pieces of the level surfaces; these can be ignored. The remaining components constitute $\mathcal{L}_-$. The first thing one notices is that $\mathcal{L}_-$ is a slice link due to the disjointness of the ascending manifolds. But, since those ascending manifolds encounter no other critical point up to the 0-level, $\mathcal{L}_-$ is actually a 0-framed unlink. (Proof: ``Watch the movie.'' As one passes index $=1$ critical levels, small ascending circles are born, the ``belt'' circles of 1-handles, and these move isotopically with previous belt circles until a new index $=1$ level is crossed. The framing condition is evident from the local Morse model at the critical point.)

	Now turn the picture upside down to see the 2-(3-)handles become 1-(0-)handles, respectively. The same argument shows that the boundary of the descending 2-manifolds (of the original) 2-handles contain a second Lagrangian $\mathcal{L}_+ \subset \Sigma$ which is also a 0-framed unlink. The 3-manifold $M$ determined by the Heegaard digram $\mathcal{D} = (\mathcal{L}_-,\mathcal{L}_+)$ evidently lies in $S^3 \times \R \subset S^4$.
	\begin{equation}
		M^3 \cong \mathrm{HB}_{\mathcal{L}_-} \smile -\mathrm{HB}_{\mathcal{L}_+}
	\end{equation}
	where $\mathrm{HB}_{\mathcal{L}_-}$ ($\mathrm{HB}_{\mathcal{L}_+}$) are the upper (lower) handlebodies.

	The Reidemeister-Singer theorem (see \cite{lau14} for a modern exposition) says that up to bar-stabilization, a given closed 3-manifold $M$ has an isotopically unique Heegaard surface. Introducing additional cancelling $(1,2)$-handle pairs to the ambient Morse function is bar-stabilization. This unique stable Heegaard decomposition has been located with an associated DU Heegaard diagram visible in $S^3 \times 0$.
\end{proof}

\begin{note}
	Using the Reidemeister-Singer theorem, only bar-stabilization and handle slides are needed to connect initial to final HDs, not Dehn twists or $\frac{1}{2}$-Dehn twists.
\end{note}

\section{Discussion}\label{sec:discussion}
Let us pause here for an informal discussion of UP/UB and to pose a question. There are many examples of closed 3-manifolds which do/don't embed in $\R^4$. Q: What happens if we quantify over the choice of boundary handlebody $\mathcal{L}_-$? Are there $f \in \operatorname{MCG}(\Sigma_g)$, $g > 1$, so that for all GLs $\mathcal{L} \subset \Sigma_g$, $(f,\mathcal{L})$, and for all $e$, $e(\lbar{\mathcal{L}})$ and $e(f(\lbar{\mathcal{L}}_-))$ can never both be 0-framed unlinks? This is the property we have called ``$f$ is UB.'' An $f$ with this purely dynamical property, if one exists, might be easier to detect.

When $g=1$, $\operatorname{MCG}(\Sigma_1) \cong \operatorname{PSL}(2,\Z)$ according to classical work of Dehn. In this case, we may check, for example, that $f = \begin{pmatrix} -2 & -3 \\ 3 & 4 \end{pmatrix}$ obeys at least an unstable form of UB, $\text{UB}_0$. That is, no conjugate $f^\pr$ of $f$ can define a gluing map yielding $M_{f^\pr} = S^1 \times D^2 \cup_{f^\pr} S^1 \times D^2$, with $M_{f^\pr}$ embedding in $S^4$. $\abs{\operatorname{torsion}(H_1(M_{f^\pr};\Z))} =: \abs{\tau(H_1(M_{f^\pr}))} =: \tau \neq 0$ since if $f^\pr = g \circ f \circ g^{-1}$, $g \in \operatorname{MCG}(S^1 \times \de D^2)$ and $g$ takes the meridian $\ast \times D^2$ to an indivisible class $[(a,b)] \neq 0 \in H_1(S^1 \times \de D^2)$, then $\tau$ may be computed as $\operatorname{det} \begin{pmatrix} a & -2a-3b \\ b & 3a+4b \end{pmatrix} = 3(a+b)^3 \neq 0$. Since $M_{f^\pr}$ has genus 1 and 3 divides $\tau$, $M_{f^\pr}$ must be a lens space and cannot embed in $S^4$. If the $f$ is bar-stabilized, many new GLs are possible. It is no longer clear that $\tau$ must be a non-square, so we do not know if $f$ is UB.

Q: If $\te \neq \mathrm{id} \in \operatorname{SP}(2g,\Z)$, can one find a Lagrangian $\mathcal{L} \in \Z^{2g}$ so that $\abs{\operatorname{torsion}(\Z^{2g}\slash \mathcal{L} \cd \te \mathcal{L})}$ is non-square? As we will see in section \ref{sec:weak-unlink}, this integer is the order of torsion$(H_1(M_f;\Z))$ for any gluing map $f$ inducing $\te$ on $H_1(\de \mathrm{HB}_\mathcal{L};\Z)$, and by Hantzsche's theorem will be a square if $M_f$ embeds in $S^4$.

If the answer to Q is ``yes,'' only an $f \in I_g := \operatorname{ker}(\mathrm{MCG}_g) \ra \operatorname{SP}(2g,\Z)$ ($I_g$ being the Torrelli group from line \eqref{eq:epimorphism}) could be UP. To give another unstable example for $g=1$, let $\te = \begin{vmatrix}
	0 & 1 \\ -1 & 0
\end{vmatrix}$ be a rotation. Let $\mathcal{L}$ be the scc $(a,b)$ ($\operatorname{gcd}(a,b) = 1$) then $\mathcal{L} = (a,b)$, $\te(\mathcal{L}) = (b,-a)$, $\abs{\operatorname{torsion}(\Z^{2g}\slash \mathcal{L} \cd \te \mathcal{L})} = \abs{\operatorname{det}\begin{pmatrix} a & b \\ b & -a\end{pmatrix}} = a^2 + b^2$, which except for Pythagorean triples is not a square.

Intuitively, UB appears to the author likely to be the generic situation for $g>1$, but it is difficult to judge: unlinks are very rare, but considerable freedom exists in the allowed choices:
\begin{enumerate}
	\item Bar-stabilization
	\item Choosing the embedding $e: \Sigma \hookrightarrow S^3$, and
	\item Choosing the GL $\mathcal{L}$
	\item The two handlebody groups applied to $\mathcal{L}$ and $f(\mathcal{L})$
\end{enumerate}

Regarding (4), consider just the effect of handle slides within the handlebody group. The example below shows that handle slides applies to a GL $\mathcal{L}$ may transform $\mathcal{L}$ back and forth between an unlink and certain ribbon links.

\begin{figure}[ht]
	\centering
	\begin{tikzpicture}
		\node at (0,0) {\includegraphics[scale=0.5]{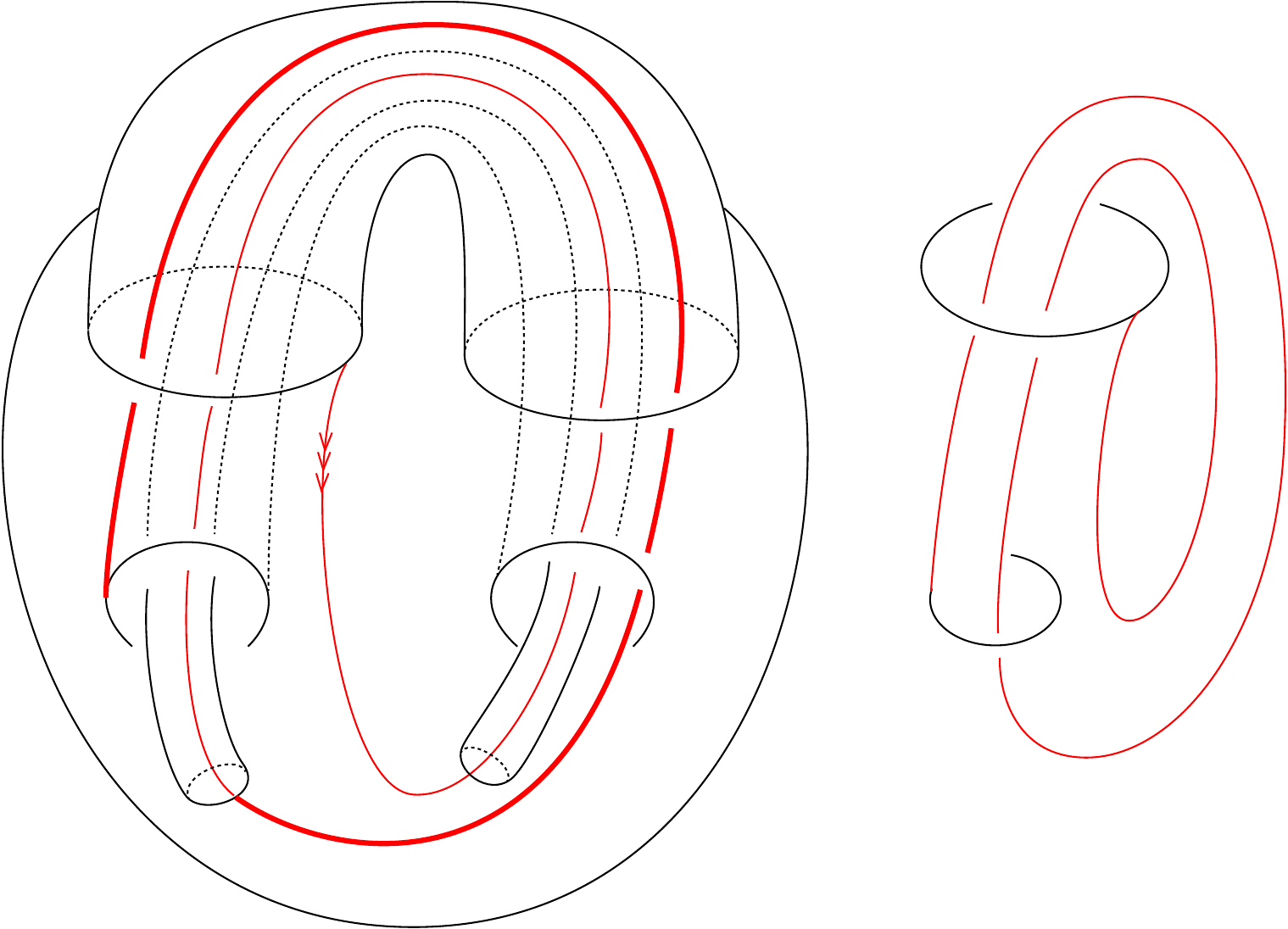}};
		\node at (-5.8,0.8) {$\mathcal{L}_3$};
		\node at (-5.75,-1.5) {$\mathcal{L}_2$};
		\node at (-3.5,-3.2) {$\mathcal{L}_1$};
		\node at (-2.7,-0.8) {$\alpha$};
		\draw[->] (-2.8,-1) -- (-3,-1.3);

		\node at (2.5,1.5) {$\mathcal{L}_3$};
		\node at (2.5,-1.5) {$\mathcal{L}_2$};
		\node at (6.8,0) {$\alpha$};
		\draw[->] (6.7,-0.2) -- (6.5,-0.5);
		\node at (4.5,-3.3) {ribbon};
	\end{tikzpicture}
	\caption{If $\mathcal{L} = (\mathcal{L}_1,\mathcal{L}_2,\mathcal{L}_3)$, an unknotted GL on the drawn genus 3 surface is modified by sliding $\mathcal{L}_3$ over $\mathcal{L}_2$ along the red arc $\alpha$, $\mathcal{L}_3$ becomes a ribbon knot. Note that it is necessary beforehand to slide $\mathcal{L}_1$ over $\mathcal{L}_2$ along the hatched portion of $\alpha$, transforming $\mathcal{L}_1$ to $\mathcal{L}_1^\pr$. So $\mathcal{L}^\pr := (\mathcal{L}_1^\pr, \mathcal{L}_2^\pr = \mathcal{L}_2, \mathcal{L}_3^\pr)$ is not an unlink but a ribbon link. Of course, the reverse slides would convert the ribbon $\mathcal{L}^\pr$ to $\mathcal{L}$.}
\end{figure}
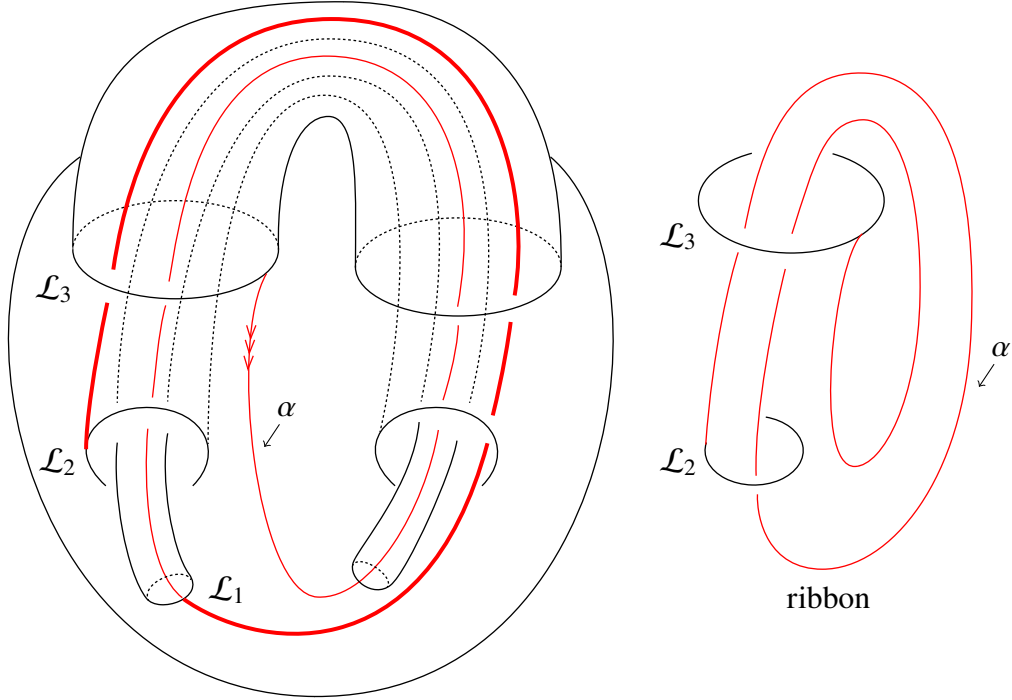

The problem of how common UB is has a $0 \cd \infty$ character: unlinks are rare, but our ``gauge freedom'' is great. However, we can whittle away a bit at the ``$\infty$'' by considering $\hat{\phantom{\_}}$-stabilization, introduced earlier. Theorem \ref{thm:hat-ub-up} produces a version of the UP/UB question equivalent to the problem of $(S^1 \times S^2)$-stabilized embeddings into $\R^4$ where the embedding $e: \Sigma_g \hookrightarrow \R^3 \times 0$ may be fixed to be the unique Heegaard surface for any fixed genus.

\begin{definition}
	A Heegaard diagram $\mathcal{D}$ is $\widehat{\mathrm{DU}}$ if, for some $\hat{\hphantom{\_}}$-stabilization $\hat{\mathcal{D}}$ there is a Heegaard embedding $e: \hat{\Sigma} \hookrightarrow S^3$ so that $\hat{R}$ and $\hat{B}$ may (separately) to 0-framed unlinks $e(h_\ell(\hat{R}))$ and $e(h_u(\hat{B}))$ for suitable $h_\ell \in \mathrm{HG}_\ell$ and $h_u \in \mathrm{HG}_u$.
\end{definition}

\begin{definition}
	$f \in \operatorname{MCG}(\Sigma_g)$ is $\widehat{\mathrm{UP}}$ iff there exists a GL $\hat{\mathcal{L}} \subset \hat{\Sigma}_g$, a $\hat{\phantom{\_}}$-stabilization of $(\Sigma_g,\mathcal{L})$ for any $k,k^\pr > 0$, so that for the unique Heegaard embedding $e: \widehat{\Sigma} \hookrightarrow S^3$, both $e(\widehat{\mathcal{L}})$ and $e(f(\widehat{\mathcal{L}}))$ are 0-framed unlinks. If $(f,\mathcal{L})$ is not $\widehat{\mathrm{UP}}$, we say it is $\widehat{\mathrm{UB}}$.
\end{definition}

Now with $M_f$, defined as before, we have:

\begin{theorem}\label{thm:hat-ub-up}
	Let $M$ be a closed 3-manifold with Heegaard diagram $\mathcal{D}$, then $M \#_{k^\pr} (S^1 \times S^2)$ for some $k^\pr$ (smoothly) embeds in $\R^4$ iff $\mathcal{D}$ is $\widehat{\mathrm{DU}}$.	In terms of a gluing map $f$, $M_{f,\mathcal{L}}$ is $\widehat{\mathrm{UP}}$ iff for some $k^\pr \geq 0$, $M_{f,\lbar{\mathcal{L}}} \# (\#_{k^\pr} S^1 \times S^2)$ smoothly embeds in $\R^4$; this embedding may be assumed Heegaard \cite{af13}, i.e.\ at all generic levels, the level surfaces are Heegaard.
\end{theorem}

\begin{proof}
	The ``if'' direction (in both versions) is the same as for Theorem \ref{thm:ub-up} after we note that $\hat{\phantom{\_}}$-stabilization has added the additional $k^\pr$ $S^1 \times S^2$ summands.

	The ``only if'' direction starts as with Theorem \ref{thm:ub-up}, by using ambient Morse theory to find, for an embedded $M_{f,\mathcal{L}}$, a mid-level Heegaard surface $\Sigma_g \subset S^3 \times 0$ between the two glued handlebodies $\mathrm{HB}_- \subset S^3 \times [0,-\infty)$ and $\mathrm{HB}_+ \subset S^3 \times [0,+\infty)$. At this point, the inclusion $\Sigma_g \subset S^3 \times 0$ may well be knotted, i.e.\ not Heegaard.

	The next lemma is well known.

	\begin{lemma}\label{lm:finite-0}
		Let $\Sigma_g \subset S^3$ be an embedded closed surface of genus $g$ in the 3-sphere. There exists a finite number of ambient 0-surgeries converting $\Sigma_g$ to a Heegaard surface $\Sigma$.
	\end{lemma}

	\begin{definition}
		An ambient 0-surgery is defined by a properly embedded arc $\alpha$ meeting $\Sigma$ exactly at its endpoints. The surgery consists of deleting a pair of disks at the endpoints and gluing in a tube around the arc. For a surface in $S^3$, such surgeries can take place on both sides.
	\end{definition}

	\begin{proof}[Proof of Lemma \ref{lm:finite-0}]
		Let $X$ denote one of the closed complementary regions of $\Sigma$ in $S^3$. $X$ has some handlebody structure rel $\Sigma$. Deleting the 1-handles (and taking closure) \emph{is} ambient 0-surgery. What is left of $X$ is a handlebody in the usual 3D topology sense, a union---in the other direction---of 0- and 1-handles.

		We are halfway there. $\Sigma$ has been surgered so it bounds a HB on one side, but the spine of this HB may be knotted in $S^3$. However, any graph may be unknotted, i.e.\ its complement made into a HB by adding some arcs. For example, looking at a planar projection, akin to a knot diagram, it suffices to add a vertical segment to the graph joining the two pre-images of every crossing point. These vertical arcs define the final set of ambient 0-surgeries (now on the other side) required to ensure that both closed complementary pieces of $\Sigma$ in $S^3$ are handlebodies. We have converted $\Sigma$ into a Heegaard surface.
	\end{proof}

	Using Lemma \ref{lm:finite-0}, stabilize $e(\Sigma_g)$ to a Heegaard surface $\widehat{\Sigma} := \Sigma_{g+k^\pr}$ by ambient 0-surgery in $S^3 \times 0$. Correspondingly, stabilize $\mathcal{L} \subset e(\Sigma_g)$ to $\widehat{\mathcal{L}} \subset \widehat{\Sigma}$, now dropping the notation $e$ for the embedding. This is done by choosing $\widehat{\mathcal{L}} \setminus \mathcal{L}$ to be the set of meridians of the new tubes. As $\hat{\hphantom{\_}}$-stabilization requires, the extension of $f$ is the identity on $\widehat{\mathcal{L}} \setminus \mathcal{L}$.

	The property that $\mathcal{L}$ ($f(\mathcal{L})$) is handlebody equivalent to an unlink is of course preserved when passing to $\widehat{\mathcal{L}}$ ($f(\widehat{\mathcal{L}})$) as any HB automorphism resulting in an unlink is merely $\hat{\hphantom{\_}}$-stabilized and still results in an unlink.

	Thus, the embedding of $M_f$ is now stabilized to an embedding of $M_{\hat{f}}$, where the property $\widehat{\mathrm{UP}}$ is now visible in the mid-level 3-sphere, $S^3 \times 0$. Line \eqref{eq:stabilize-manifold} identifies $M_{\hat{f}}$ with $M_f \# (\#_{k^\pr} S^1 \times S^2)$. To complete the proof, just as in Theorem \ref{thm:ub-up}, the Reidemeister-Singer theorem \cite{lau14} is applied to identify input and output Heegaard surfaces, possibly requiring an additional bar-stabilization, which does not affect the topology of the embedded 3-manifold.
\end{proof}

The ordered ambient Morse function on $M$ is close to the concept of a 4-section, and becomes identical in the $\hat{\hphantom{\_}}$-stabilized context. Reproducing Definition 2.1 of \cite{in20}:
\begin{definition}
	Let $X$ be a smooth, orientable, closed, connected 4-manifold. An \emph{n-section}, or \emph{multisection}, of $X$ is a decomposition $X = X_1 \cup \cdots \cup X_n$ such that:
	\begin{enumerate}
		\item $X_i \cong \#^{k_i} S^1 \times B^3$;
		\item $X_1 \cap \cdots \cap X_n = \Sigma_g$, a closed orientable surface of genus $g$;
		\item $X_i \cap X_j = H_{i,j}$ is a 3-dimensional handlebody if $\abs{i-j} = 1$ and $X_i \cap X_j = \Sigma_g$ if $\abs{i-j} > 1$, where $i,j$ are treated cyclically as elements of $\Z_n$;
		\item $\de X_i \cong \#^{k_i} S^1 \times S^2$ has a Heegaard splitting given by $H_{(i-1),i} \cup_\Sigma H_{i,(i+1)}$.
	\end{enumerate}
\end{definition}

In \cite{af13} we defined a \emph{Heegaard embedding} of a closed $M^3$ in $S^4$ to be an embedding so that when considered to lie in $S^3 \times \R \cong S^4 \setminus (\text{2 pts.})$ the $\R$-coordinate is ordered Morse and in every generic level $t$, $M^3 \cap (S^3 \times t) \subset S^3 \times t$ is a Heegaard surface. We proved (Theorem 4.1 \cite{af13}) that every $M^3 \hookrightarrow S^4$ with a unique local min (or a unique local max) is isotopic to a Heegaard embedding and gave examples (section 2 \cite{af13}) of embedding $M^3 \hookrightarrow S^4$ having multiple local mins and maxes which cannot be reduced by any isotopy.

In modern language, Heegaard embedding is a 4-section with the four 3D handlebodies being, in cyclic order: upper $\operatorname{HB}(M^3)$, outer $\operatorname{HB}(S^3)$, lower $\operatorname{HB}(M^3)$, and inner $\operatorname{HB}(S^3)$. In the same paper we show that after $\hat{\hphantom{\_}}$-stabilization, any embedding $M^3 \hookrightarrow S^4$ can be converted to $M^3 \# (\#_k S^1 \times S^2) \hookrightarrow S^4$ which is isotopic to a Heegaard embedding (a 4-section). For some embeddings $M \hookrightarrow S^4$ we show there is no isotopy to a Heegaard embedding, but it is an \emph{open question}: Given $M^3 \overset{f}{\hookrightarrow} S^4$, can one find a new embedding $M^3 \overset{f^\pr}{\hookrightarrow} S^4$ which is Heegaard (a 4-section), without any stabilization? But even without knowing if this refinment is possible, the ordered-Morse embedding we have constructed allows the enhancement of Hantzsche's theorem (see section \ref{sec:weak-unlink}).

\section{Linking Forms and Further Embedding / Non-embedding Examples}\label{sec:examples}

Before turning to Hantzsche's theorem, consider an even more elementary obstruction to a closed 3-manifold $M$ embedding in $S^4$.

\begin{fact}\label{fact:dne}
	$M$ does not embed in $S^4$ if $M$ contains a closed embedded surface $S$ so that $\langle S \rangle \neq 0 \in \Omega_2 \cong \Z_2$, the unoriented bordism group in dimension 2.
\end{fact}

\begin{proof}
	The non-bounding surfaces are those with odd Euler characteristic, $\#_{i=1}^{\text{odd}} RP_i^2$; the rest bound. It is a classical result of Whitney \cite{whit40} that if a non-bounding $S$ is embedded in $S^4$ its normal bundle has twisted Euler class $\equiv 2$ mod 4. In particular, the normal bundle is \emph{not} a line bundle stabilized by a trivial summand. But, if there were an embedding $M \subset S^4$, then the normal bundle $\cup_{S \hookrightarrow S^4}$ of $S \subset M \subset S^4$ would have precisely that structure.
\end{proof}

It is convenient to have an algebraic version of this obstruction. Any non-bounding $S \subset M$ determines a class $[S] \neq 0 \in H_2(M; \Z_2) \cong H^1(M; \Z_2)$. Such elements $[S]$ are precisely the classes which cube to the generator of $H^3(M;\Z_2)$, or, equivalently, have triple intersection $=1$. To see this, consider the diagram for the classifying map:
\begin{figure}[ht]
	\centering
	\begin{tikzpicture}
		\node at (0.1,0) {$M \longrightarrow RP^3 \lhook\joinrel\longrightarrow RP^\infty = K(\Z_2,1)$};
		\node at (-1.7,-1.2) {$S \longrightarrow RP^2$};
		\node[rotate=90] at (-2.4,-0.6) {$\hookrightarrow$};
		\node[rotate=90] at (-1.2,-0.6) {$\hookrightarrow$};
		\node at (-2.45,-2.4) {$C := S \pitchfork S \longrightarrow RP^1$};
		\node[rotate=90] at (-2.4,-1.8) {$\hookrightarrow$};
		\node[rotate=90] at (-1.2,-1.8) {$\hookrightarrow$};
	\end{tikzpicture}
	\caption{}
\end{figure}

The ``double curve'' $C = S \pitchfork S$ is always P.D.\ to $\omega_1(\tau_s)$ (since $\nu_{C \hookrightarrow S} \cong \nu_{S \hookrightarrow M}\vert_C$), and $C$ reverses orientation on $S$ iff $S$ is non-bounding. But $C$ reverses orientation on $S$ iff $S \cd S \cd S = 1$ (mod 2). Let us define the map $\te: H_2(M; \Z) \ra \Z_2$ given by a triple product, $\te([S]) = S \cd S \cd S$ mod 2. \qed

\begin{fact}
	$\te: H_2(M; \Z_2) \ra \Z_2$ sends $[S]$ to 0 (1) iff $\langle S \rangle = 0\ (=1) \in \Omega_2$. So, by Fact \ref{fact:dne}, $\te \equiv 0$ if $M$ embeds in $S^4$. Furthermore, $\te$ is a group homomorphism.
\end{fact}

\begin{proof}
	To check the last statement, let $S$ and $S^\pr$ represent (possibly) distinct classes in $H_2(M; \Z_2)$ and now $C := S \pitchfork S^\pr$. Because $M$ is orientable, for any component $C_0 \subset C$, neighborhoods of $C_0 \subset S$ and $C_0 \subset S^\pr$ are either both annuli or both M\"{o}bius bands. So, the processes of surgering out double curves preserves Euler characteristic, and hence the element in $\Omega_2$.
\end{proof}

\subsection*{Linking forms}
Linking forms are symmetric non-singular pairings $\lambda: \tau \times \tau \ra \Q \slash \Z$, where we abbreviate $\operatorname{torsion}(H_1(M;\Z))$ by $\tau$. When $M$ has no 2-torsion, the pairing takes values in $\Z_{\text{odd}} \slash \Z$, where $\Z_{\text{odd}}$ denotes the integers with the odd primes inverted, and in this case, row and column operators over $\Z$ allow us to find a basis in which $\lambda$ is diagonal.

\subsection*{Aside on prime 2}
For $M = RP^3 \# RP^3$, $\tau = Z_2 \oplus \Z_2$, and $\lambda = \begin{vmatrix}
	\frac{1}{2} & 0 \\ 0 & \frac{1}{2}
\end{vmatrix}$, in the natural basis a form \emph{cannot} be hyperbolized over $\Z$. However, if $M$ is the $T^2$ bundle over $S^1$ with monodromy $\begin{vmatrix}
	-1 & 0 \\ 0 & -1
\end{vmatrix}$, call it $B$, then $H_1(B) \cong \Z \oplus \Z_2 \oplus \Z_2$, $\tau = \Z$, and $\lambda = \begin{vmatrix}
	0 & \frac{1}{2} \\ \frac{1}{2} & 0
\end{vmatrix}$, where we use the basis consisting of the fibers of the two ``horizontal'' Klein bottles in $B$ (these intersect the fiber $T^2$ in a meridian and longitude, respectively). This linking form is \emph{not} diagonalizable over $\Z$. It is known \cite{epstein65} no $L_{p,q}^-$, the punctured lens space, $p$ even, embeds in $S^4$, so in particular $RP^3 \# RP^3 = L_{2,1} \# L_{2,1}$ does not embed in $S^4$. By \cite{zee65}, all $L_{p,q}^-$, $p$ odd, do embed in $S^4$. We see $(RP^3)^-$ \emph{cannot} embed due to the implied splitting of the normal bundle $RP^2 \subset (RP^3) \subset S^4$ (see Fact \ref{fact:dne}), but $H_2(L_{4,1}^-;\Z_2)$ is generated by an embedded Klein bottle, so its exclusion goes beyond Fact \ref{fact:dne}. The bundle $B$, as its hyperbolic form might suggest, does embed in $S^4$. To see the embedding, initially embed $M \hookrightarrow S^1 \times B^3$ fiber-wise by embedding $T^2 \subset B^3$ and revolving $B^3$ by $2\pi$ around an approximate axis of symmetry for $T^2$, then embed $S^1 \times B^3 \hookrightarrow S^4$.

A genus 3 HD for $B$ is easily derived from the fundamental domain, shown in Figure \ref{fig:b-hd}.
\begin{figure}[ht]
	\centering
	\includegraphics[scale=0.4]{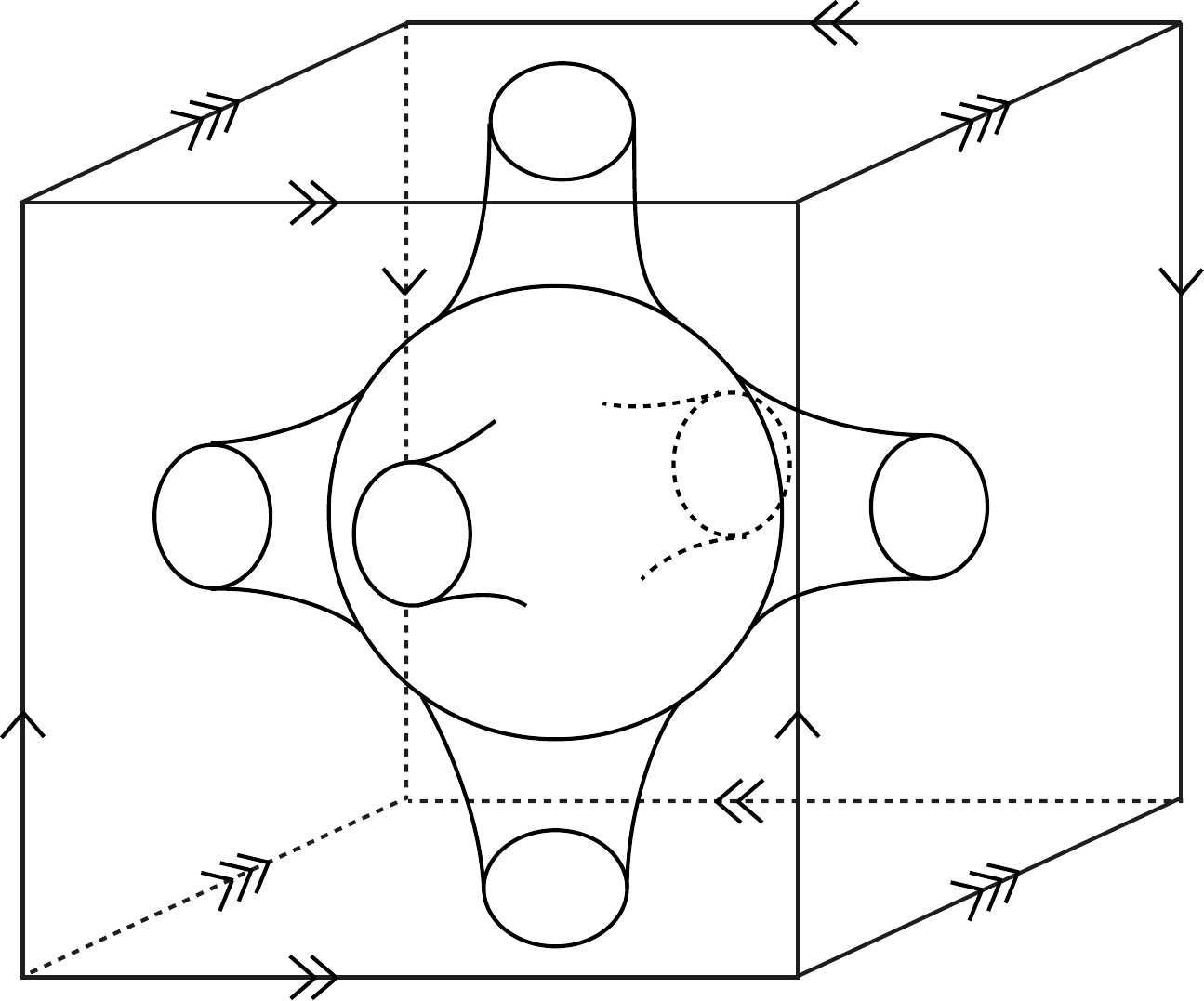}
	\caption{}\label{fig:b-hd}
\end{figure}

Homologically, the intersection matrix of $\mathcal{L}_+$ with $\mathcal{L}_-$, i.e.\ the boundary map between 1 and 2 handles, is $\begin{vmatrix}
	0 & 0 & 0 \\ 0 & 0 & 2 \\ 0 & 2 & 0
\end{vmatrix}$.

\section{Weakening the Unlink Condition and Enhancing Hantzsche}\label{sec:weak-unlink}
We have seen (Theorem \ref{thm:ub-up} ``only if'') that after isotoping an embedding $M^3 \hookrightarrow \R^4$ so that the 4th coordinate becomes an ordered Morse function, a pair of 0-framed unlinks $\mathcal{L}_-$ and $\mathcal{L}_+$ appears in the Heegaard surface $\Sigma$, for $M$ at the middle level $\R^3 \times 0$, or, equivalently, $S^3 \times 0$, $(\mathcal{L}_-,\mathcal{L}_+)$ being equivalent notation for the pair of Lagrangians $(R,B)$.

Suppose we ``forget'' that $\mathcal{L}_\pm$ are unlinks, and remember only that within $\mathcal{L}_-$ ($\mathcal{L}_+$) pairwise linking numbers and framings vanish. Then we may derive the conclusion of Hantzsche's theorem from this weaker hypothesis. First, we set up and prove the original theorem. For closed 3-manifolds, Poincar\'{e} duality gives an isomorphism:
\begin{equation}\label{eq:iso-1}
	H_1(M;\Z) \cong H^2(M;\Z),
\end{equation}
and $H^2$ fits in a universal coefficient sequence:
\begin{equation}\label{eq:iso-2}
	0 \ra \operatorname{Ext}_\Z(H_1(M;\Z),\Z) \ra H^2(M;\Z) \ra \operatorname{Hom}(H_2(M;\Z),\Z) \ra 0,
\end{equation}
with
\begin{equation}\label{eq:iso-3}
	\operatorname{Ext}_\Z(H_1(M;\Z),\Z) \text{ naturally isomorphic to } \operatorname{Hom}_\Z(H_1(M;\Z),\Q\slash\Z)
\end{equation}
Combining \eqref{eq:iso-1}, \eqref{eq:iso-2}, and \eqref{eq:iso-3} and using $\tau$ to denote the torsion subgroup, there is a natural isomorphism $\tau H_1(M;\Z) \cong \operatorname{Hom}(\tau H_1(M;\Z),\Q \slash \Z)$, which may be interpreted as the \emph{nonsingular} linking form:
\begin{equation}
	\tau H_1(M;\Z) \otimes \tau H_1(M;\Z) \xrightarrow{\text{LF}} \Q \slash \Z
\end{equation}

\begin{theorem}[Hantzsche]\label{thm:hantzsche}
	If $M^3$ embeds in $S^4$ there is a splitting $\tau H_1(M;\Z) \cong A \oplus B$ with LF vanishing identically on both $A$ and $B$. Since $\mathrm{LF}$ is nonsingular, this immediately gives natural isomorphism $A \cong \operatorname{Hom}(B, \Q\slash\Z)$ and $B \cong \operatorname{Hom}(A,\Q\slash\Z)$. In particular, $\abs{\tau H_1(M;\Z)}$ is a square integer. We say that $\operatorname{LF}(M)$ is \emph{hyperbolic}.
\end{theorem}

\begin{proof}
	Write $S^4 = \mathcal{A} \cup_M \mathcal{B}$ and define $A = \ker(\tau H_1(M;\Z) \ra H_1(\mathcal{A};\Z))$ and $B = \ker(\tau H_1(M;\Z)$ $\ra H_1(\mathcal{B};\Z))$. The usual proof proceeds by showing that the linking pairing on $M$ vanishes when restricted to the kernel into the 4-manifolds $\mathcal{A}$ or $\mathcal{B}$ in $S^4$ that $M^3$ bounds.

	Taking always $\Z$-coefficients, let $p,q \in \ker(\tau H_1(M) \ra H_1(\mathcal{A}))$ and let $t$ be order $\tau H_1(M)$. Let $tp = \de u$ and $tq = \de v$. The linking number $\lambda(p,q)$ is computed:
	\begin{equation}
		\lambda(p,q) = \frac{1}{t} p \cd v = \frac{1}{t} u \cd q = \frac{1}{t^2} X \cd Y
	\end{equation}
	where $X = t \delta_{\mathcal{A}}(p) - u$, $Y = t \delta_{\mathcal{A}}(q) - v$ are 2-cycles in $\mathcal{A}$ naturally associated to $p,q$. $\delta_{\mathcal{A}}(p)$ denotes a 2-chain in $\mathcal{A}$ with boundary $p$, similarly for $\delta_{\mathcal{B}}(q)$. But since $\mathcal{A} \subset S^4$ the intersection form on $A$ is trivial, $X \cd Y = 0$, showing $A$ and $B$ are isotropic subspaces of the linking form $\lambda$.
\end{proof}

\begin{figure}[ht]
	\centering
	\begin{tikzpicture}[scale=1.1]
		\draw (0,0) circle (4);
		\draw[rotate=-30] (0,-4) arc(270:90:0.7 and 4);
		\draw[rotate=-30,dashed] (0,-4) arc(-90:90:0.7 and 4);
		\draw (-4,0) arc (180:360:4 and 0.7);
		\draw[dashed] (-4,0) arc (180:0:4 and 0.7);
	
		\draw[fill=black] (-1.17,-0.67) circle (0.3ex);
		\node at (-1.6,-0.9) {$\Sigma$};
		\draw[fill=black] (1.17,0.67) circle (0.3ex);
		\node at (1.6,0.9) {$\Sigma$};
	
		\node at (-1.5,2) {$\mathcal{A}$};
		\node at (2.5,1.8) {$\mathcal{B}$};
		\node at (-2.5,-1.8) {$\mathcal{A}$};
		\node at (1.5,-2) {$\mathcal{B}$};
	
		\node at (-4.3,0) {$\alpha$};
		\node at (4.3,0) {$\beta$};
		\node at (0.5,2.7) {$\text{HB}_+$};
		\node at (-2.4,-2.55) {$\text{HB}_-$};
	
		\node at (3.5,-1.3) {$S^3 \times 0 = \alpha \cup_\Sigma \beta$};
		\draw[->] (3.5,-1.1) -- (3.2,-0.55);
		\node at (3.1,4.1) {$M = \mathrm{HB}_+ \cup_\Sigma \mathrm{HB}_-$};
		\draw[->] (3.1,3.8) to[out=225,in=15] (1.5,3);

		\node at (-4.6,0) {$\hphantom{m}$};
	\end{tikzpicture}
	\caption{}\label{fig:isotropic-surface}
\end{figure}

By Theorem \ref{thm:ub-up}, we may replace the embedding hypothesis in Theorem \ref{thm:hantzsche} with the hypothesis that the Heegaard diagram HD is DU, but actually, a weaker hypothesis suffices.

\begin{definition}
	A Heegaard diagram is homologically doubly unlinked (HDU) iff its pair of GLs $(\mathcal{L}_+. \mathcal{L}_-)$ for $\lbar{\Sigma}$ a bar-stabilization of $\Sigma_g$ (some $k \geq 0$), and a choice of embedding $\lbar{\Sigma} \overset{e}{\hookrightarrow} \Sigma^3$, $\Sigma^3$ an integral homology 3-sphere, so that under this embedding both $e(\mathcal{L}_+)$ and $e(\mathcal{L}_-)$ are handlebody equivalent on $e(\lbar{\Sigma})$ to links with vanishing linking numbers---including the diagonal framings.
\end{definition}

With this definition we state an enhanced Hantzsche theorem:

\begin{theorem}\label{thm:enhanced-hantz}
	If a closed 3-manifold $M$ has a HDU Heegaard diagram $\mathcal{D}$, then $\tau H_1(M)$ splits as $A \oplus B$ with the LF vanishing identically on both $A$ and $B$, i.e.\ iff the LF is hyperbolic. Conversely, if $M$ is a $\Z_2$-homology 3-sphere with hyperbolic LF, then $M$ has a HDU Heegaard diagram.
\end{theorem}

\begin{proof}
	The first statement is similar to the proof of Theorem \ref{thm:hantzsche}, but instead of working $S^4$ we build a bespoke 4-manifold $\Z$ in which certain linking numbers in $M$ may be computed in terms of intersection numbers of 2-cycles. $\Z$ may be pictured as something like Figure \ref{fig:isotropic-surface} with voids where the four ``quadrants'' are seen. Similar to writing $S^4 = \mathcal{A} \cup_M \mathcal{B}$, we will write $\Z = \mathfrak{A} \cup_M \mathfrak{B}$.

	To be explicit, $\Sigma^3 = \alpha \cup_{\Sigma^2} \beta$, the union of the complementary regions of $e(\Sigma^2)$. Thickening, we write:
	\begin{equation}
		\Sigma^3 \times [-1,1] = \alpha \times [-1,1] \cup_{\Sigma^2 \times [-1,1]} \beta \times [-1,1].
	\end{equation}
	Form $\Z$ by attaching a $\mathrm{HB}_u \times [-\epsilon,\epsilon]$ to $(\Sigma^2 \times 1) \times [-\epsilon,\epsilon]$ so that $f(\mathcal{L}_+) \times 0$ bounds disks in $\mathrm{HB}_u \times 0$ and by attaching $\mathrm{HB}_\ell \times [-\epsilon,\epsilon]$ to $(\Sigma^2 \times 1) \times [-\epsilon,\epsilon]$ so that $e(\lbar{\mathcal{L}}_-) \times 0$ bounds disks in the lower handlebody $\mathrm{HB}_\ell \times 0$. The interval $[-\epsilon,\epsilon]$ parameterizes a normal bi-collar to $e(\Sigma) \hookrightarrow \Sigma^3$. The pieces $\mathfrak{A}, \mathfrak{B} \subset \Z$ are
	\begin{equation}
		\begin{split}
			\mathfrak{A} & = \alpha \times [-1,1] \cup \mathrm{HB}_u \times [-\epsilon,0] \cup \mathrm{HB}_\ell \times [-\epsilon,0]\text{, and } \\
			\mathfrak{B} & = \beta \times [-1,1] \cup \mathrm{HB}_u \times [0,\epsilon] \cup \mathrm{HB}_\ell \times [0,\epsilon].
		\end{split}
	\end{equation}
	Of course, $M = \mathrm{HB}_u \cup_{\lbar{\Sigma}} \mathrm{HB}_\ell$.

	Again, let $\tau H_1(M)$ and denote $A := \operatorname{ker}(\tau H_1(M) \ra H_1(\mathfrak{A}))$ and $B := \operatorname{ker}(\tau H_1(M) \ra H_1(\mathfrak{B}))$. Just as in the proof of Theorem \ref{thm:hantzsche}, let $t = \abs{\tau H_1(M)}$, $p,q \in A$, $tp = \de u$, $tq = \de v$. Following \cite{gl78} (see pages 59 and 60), the linking number $\lambda(p,q)$ is computed:
	\begin{equation}
		\lambda(p,q) = \frac{1}{t} pv = \frac{1}{t} uq = \frac{1}{t^2} X \cd Y,
	\end{equation}
	where $X = t \delta_{\mathfrak{A}} - u$, $Y = t \delta_{\mathfrak{A}}(q) - v$ are 2-cycles in $\mathfrak{A}$ naturally associated to $p$ and $q$.

	The manifold $\mathfrak{A}$ has an identically vanishing intersection form over the integers: $H_2(\mathfrak{A}) \times H_2(\mathfrak{A}) \xrightarrow{\text{zero}} Z$, and similarly $H_2(\mathfrak{B}) \times H_2(\mathfrak{B}) \xrightarrow{\text{zero}} Z$. These vanishings are a consequence of the vanishing $H_2(\Z) \times H_2(\Z) \xrightarrow{\text{zero}} Z$ which we now verify. Consider the ``vertical'' decomposition of $\Z$
	\begin{equation}\label{eq:decomp-z}
		\Z = U \cup_{\Sigma^3 \times 0} D
	\end{equation}
	where ``up'' $= U = \Sigma^3 \times [0,1] \cup \mathrm{HB}_u \times [-\epsilon,\epsilon]$, and ``down'' $= D = \Sigma^3 \times [-1,0] \cup \mathrm{HB}_\ell \times [-\epsilon,\epsilon]$.

	The hypothesis that $e(\mathcal{L}_+)$ has (up to an automoprhism of $\mathrm{HB}_u$) vanishing linking and self-linking numbers means $H_2(\mathrm{U}) \times H_2(\mathrm{U}) \xrightarrow{\text{zero}} Z$ is identically zero. Similarly, the hypothesis on $e(\mathcal{L}_-)$ implies that $H_2(D) \times H_2(D) \xrightarrow{\text{zero}} Z$ vanishes. Now exploit the fact that \eqref{eq:decomp-z} $\Z$ is made from $U$ and $D$ by gluing along a homology sphere to conclude that
	\begin{equation}
		H_2(\Z) \times H_2(\Z) \xrightarrow{\text{zero}} Z
	\end{equation}
	vanishes. Figure \ref{fig:decomp-z} summarizes the dance between ``horizontal'' and ``vertical'' decompositions of $\Z$ on which the proof depends.

	\begin{figure}[ht]
		\centering
		\begin{tikzpicture}
			\node at (0,0) {\includegraphics[scale=0.6]{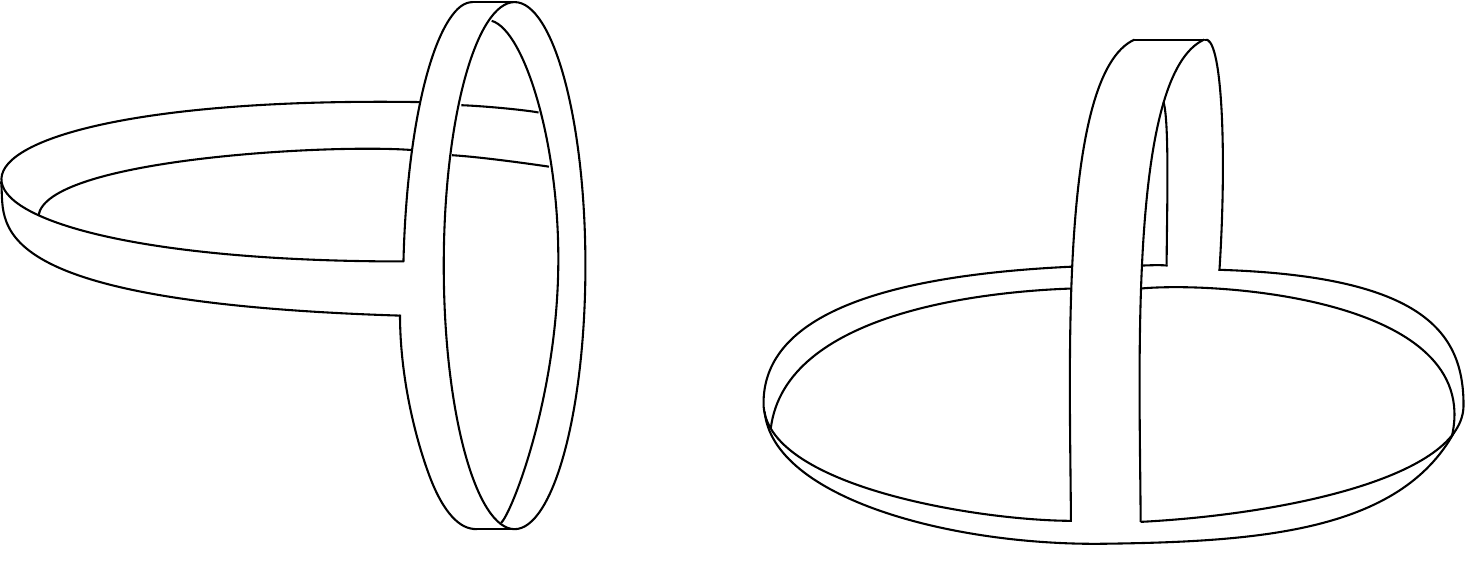}};
			\draw[decorate,decoration={brace,amplitude=5pt}] (-7.7,-3) -- (-7.7,3) node[midway,xshift=-0.5cm] {$\mathfrak{A}$};
				
			\draw[<->] (-2.75,-2.7) -- (-2.25,-2.7);
			\node at (-2.5,-3) {$[-\epsilon,0]$};
			\node at (-4.5,-2) {$\mathrm{HB}_\ell \times [-\epsilon,0] \ra$};
			
			\node at (-5.5,-0.8) {$\alpha \times [-1,1]$};
			\draw[->] (-4.7,-0.5) -- (-4.5,-0.3);
			\node at (-4.6,2.7) {$\mathrm{HB}_u \times [-\epsilon,0] \ra$};
			\node at (-1.2,2.5) {$\leftarrow M$};
		
			\draw[decorate,decoration={brace,amplitude=5pt}] (7.5,-2.8) -- (0.1,-2.8) node[midway,yshift=-0.5cm] {$\mathrm{U}$};
			\node at (1.8,2) {$\mathrm{HB}_u \times [-\epsilon,\epsilon]$};
			\draw[->] (3.1,1.9) -- (3.3,1.7);
			\node at (1.8,-1.5) {$\Sigma^3 \times [0,1]$};
			\draw[->] (2,-1.8) -- (2.2,-2.1);
		\end{tikzpicture}
		\caption{}\label{fig:decomp-z}
	\end{figure}

	For the converse direction, assume $M$ is a $\Z_2$-homology 3-sphere and LF$(M)$ is hyperbolic (vanishing on $A$ and $B$ with $A \oplus B = \tau H_1(M;\Z)$). It is helpful to put LF$(M)$ into a normal, diagonal, form. $\mathcal{L}_-$ is the meridians to a lower genus $g$ handlebody $\mathrm{HB}_-$. A standard application of the Euclidean algorithm (as stated in section \ref{sec:examples}) allows us to slide and permute the 2-handles of $\mathrm{HB}_+$, attached along $\mathcal{L}_+$, to make the homological boundary map between and 2- and 1-handles diagonal with all 0 or 1 diagonal entries in the upper left corner. The number of zeros is $r := \mathrm{rank}(H_1(M;\Q))$. The number $s$ of 1 entries is the number of homologically cancelling handle pairs, and the other diagonal entries are the order of the summands of a cyclic decomposition of $\tau H_1(M)$ with $g-r-s$ summands. This much is visible just from the image of $\mathcal{L}_+$ in $H_1(\mathrm{HB}_-)$; to compute the linking form we also need to know the image of $\mathcal{L}_+$ in $H_1(\Sigma) = H_1(\de \mathrm{HB}_-) \cong \Z^{2g} = \operatorname{span}(\text{1-handle co-cores, }\mathcal{L}_-\text{, and some collection of homological }\delta_{ij}\text{-duals }\hat{\mathcal{L}}_-)$. Indeed, if $a,b \in \tau H_1(M)$, then we may write:
	\begin{multline}
		ta = x+y,\ tb = w+z,\ t = \abs{\tau H_1(M)},\ x,w \in \operatorname{ker}(H_1(\Sigma) \ra H_1(\mathrm{HB}_-)) \\ \text{ and } y,z \in \operatorname{ker}(H_1(\Sigma) \ra H_1(\mathrm{HB}_+)),
	\end{multline}
	then:
	\begin{equation}
		\text{linking number}(a,b) =: \lambda(a,b) = \frac{1}{t}\langle y,b \rangle_\Sigma = \frac{1}{t}\langle z,a\rangle_\Sigma,
	\end{equation}
	$\langle, \rangle_\Sigma$ the intersection number.

	If we use a $(\mathcal{L}_i,\hat{\mathcal{L}}_i)$-basis for $H_1(\Sigma_g)$, the $\de$-map from $H_1(\mathcal{L}_+)$ is diagonal $\langle \mathcal{L}_{-,i},\mathcal{L}_{+,i} \rangle_\Sigma = p_i$, a diagonal entry of the $\de$-map to $H_1(\text{lower HB})$. Then, there is a well-defined $q_i$, $-p_i < q_i < p_i$, $q_i = \langle \hat{\mathcal{L}}_{-,i},\mathcal{L}_{+,i}\rangle_\Sigma$, and $q_i \slash p_i$ is the $i$th entry of the linking form, $g-r-s \leq i \leq q$.

	Because $M$ is assumed to be a $\Z_2$-homology 3-sphere, the $p_i$ are odd. Also, from this hypothesis (see section \ref{sec:examples}), the linking form is diagonalizable (the example $B$ was introduced to show a non-diagonalizable LF when 2-torison is present). Using $\mathcal{L}_-$ and $\hat{\mathcal{L}}_-$ as a basis for $H_1(\Sigma_g)$, the homological image of $\mathcal{L}_+$ in $H_1(\Sigma_g)$ is determined, up to the action of the lower handlebody group, by the $r$ 0's, $s$ 1's, and $g-r-s$ rationals $\frac{q_i}{r_i}$. If the Heegaard gluing of the 3-manifold $M$ is multiplied (left or right) by any element of the Torelli subgroup to make a new manifold $M^\pr$, then the linking form is unchanged, $\operatorname{LF}(M^\pr) = \operatorname{LF}(M)$. The Torelli group acts transitively on $\{\mathcal{L}_+\}$ with the same homological image in $H_1(\Sigma_g)$. Furthermore, the Torelli action on embedded surfaces $\Sigma_g \subset S^3 \times 0$ will also leave all (self-)linking numbers of a sublink $\mathcal{L}_\pm \subset \Sigma_g$ unchanged.

	This means if we have a single example of an $M \hookrightarrow S^4$ (implying a DU Heegaard diagram for $M$) with LF$(M) = \lambda$, then any $M^\pr$, as above, must have at least a HDU Heegaard diagram. The examples
	\begin{equation}
		L = \#_i (L_{p_i,q_i} \# -L_{p_i,q_i}),
	\end{equation}
	since they all embed in $S^4$ \cite{zee65}, and cover all LF$(\Z_2-\text{homology sphere})$, showing that every $\Z_2$-homology sphere has a HDU Heegaard diagram embedded in $S^3$.

	It is worth noting that integral homology 3-spheres trivially have a hyperbolic linking form (in this case, after handle slides the boundary map becomes the identity on homology) so all Heegaard diagrams of integral homology spheres are HDU.
\end{proof}

\begin{note}
	This partial converse can be extended to show that if $(\mathcal{L}_-,\mathcal{L}_+)$ is associated to an embedding $M \subset S^4$ and $\mathcal{L}_+$ is acted on by an element of the $k$th Johnson subgroup \cite{john83} to get $\mathcal{L}_+^\pr$, then $M^\pr$ defined by $(\mathcal{L}_-,\mathcal{L}_+^\pr)$ will ($2_k$)-embed (see the note below for the definition).
\end{note}

\begin{note}
	We have seen that linking numbers within $e(\mathcal{L}_+)$ and $e(\mathcal{L}_-)$ provide exactly the information to recover the conclusion of Hantzsche's theorem. Between linking number and unlink are many computable invariants, e.g.\ the higher order Milnor $\lbar{\mu}$-invariants. A virtue of using an ordered embedding to define the concept of unlink preserving (UP) is it allows a unified way of conceptualizing myriad conditions weaker than an embedding $M^3 \hookrightarrow S^4$ but, in a sense, converging toward an embedding.
\end{note}

Let $P$ stand for some property of a framed link $L$ in $S^3$ Examples of interest include:

	(1) Vanishing of linking and self-linking numbers

	($2_k$) Vanishing of all $\lbar{\mu}$-inv of length $\leq k$ (when $k = 2$ this reduces to the previous example)

	(3) Boundary link

	(4) Slice link

	(5) Ribbon link

	(6) Has Kauffman bracket agreeing with the unlink with the same number of components

	(7) Cannot be distinguished from the unlink by representations into a fixed group $G$

	(8) Unlink (meaning 0-framed unlink)

\begin{definition}
	We say a closed 3-manifold $M$ \emph{$P$-embeds} (in $S^4$) iff $M$ has a Heegaard diagram $(\mathcal{L}_+,\mathcal{L}_-) \subset \Sigma$ so that there is an embedding $\Sigma \overset{e}{\hookrightarrow} S^3$ (not necessarily as a Heegaard embedding) so that $e(\mathcal{L}_+)$ and $e(\mathcal{L}_-)$ separately enjoy the property $P$.
\end{definition}

Theorem \ref{thm:ub-up} says: $M$ embeds in $S^4 \iff M$ $P(8)$-embeds. Theorem \ref{thm:enhanced-hantz} says LF$(M)$ hyperbolic $\iff M$ $P(1)$-embeds. Perhaps there is some property $P$, weaker than ``unlink'' but more tractable, that would allow us to search for $M$'s embeddable in homotopy 4-spheres but which do not $P$-embed.

Consider $P(2_k)$ above, $k = 3,4,\dots,\infty$. We can imagine finding a closed 3-manifold $M$ which embeds some homotopy 4-sphere but does not $P(2_k)$-embed. The virtue of property $P(2_k)$ is that, although much weaker than $P(8)$ ``unlink,'' its behavior under indeterminacies should be easier to track.

\bibliography{references}

\end{document}